\documentclass[12pt,leqno]{article}

\columnwidth\textwidth

\usepackage{amsmath}
\usepackage{amsthm}
\usepackage{amssymb}
\usepackage{amsfonts}
\usepackage{latexsym}
\usepackage{indentfirst}
\usepackage{graphicx}
\usepackage{color}
\usepackage[english]{babel}
\usepackage{paralist}
\usepackage{fullpage}
\usepackage{color}
\usepackage{stmaryrd}
\usepackage{amsfonts}
\usepackage{comment}
\date{}

    \newcommand{\cb}{{\mathcal B}}

\numberwithin{equation}{section}

\providecommand\mathbb{\bf}

\providecommand\text[1]{\textrm{#1}}

\newtheorem{thm}{Theorem}[section]
\newtheorem{lem}[thm]{Lemma}
\newtheorem{prop}[thm]{Proposition}

\newtheorem{rem}[thm]{Remark}  

\newtheorem{exam}[thm]{Example} 

\numberwithin{equation}{section}

\begin{document}

\begin{center}
{\large Nonlinear Dirichlet problem of non-local branching processes}\\[3mm]
Lucian Beznea\footnote{Simion Stoilow Institute of Mathematics of the Romanian Academy, P.O. Box 1-764, 014700
	Bucharest, Romania, and University POLITEHNICA Bucharest, CAMPUS Institute, Bucharest, Romania. 
 {E-mail}:
{lucian.beznea@imar.ro}},
{Oana Lupa\c{s}cu-Stamate}\footnote{Institute of Mathematical Statistics and Applied Mathematics of the Romanian Academy,  
Calea 13 Septembrie 13, Bucharest, Romania.
 E-mail: oana.lupascu@yahoo.com}, 
and
{Alexandra Teodor}\footnote{University POLITEHNICA Bucharest, Bucharest, Romania, and Simion Stoilow Institute of Mathematics of the Romanian Academy.
{E-mail}: alexandravictoriateodor@gmail.com}
\end{center}

\vspace{5mm}

\begin{abstract}
	We present a method of solving a nonlinear Dirichlet problem with discontinuous boundary data and  we give a probabilistic representation of the solution using the non-local branching process associated with the nonlinear term of the operator. Instead of the pointwise convergence of the solution to the given boundary data we use the controlled convergence which allows to have discontinuities at the boundary.  

\end{abstract}

\vspace{4mm}




\vspace{1mm}


\section{Introduction}

Consider the following nonlinear Dirichlet problem%
\begin{equation}  \label{problem}
\left\{ 
\begin{array}{c}
( \Delta -c) u+c\underset{k\geqslant 1}{\sum }b_{k}B_{k}u^{(
k) }=0\text{ in } D \\[4mm]
\underset{D \ni x \rightarrow y}{\lim} u(x)=\varphi(y) \text{ for all }y \in \partial D,%
\end{array}%
\right.  
\end{equation}%
where $D$ be a  bounded domain of  $\mathbb{R}^d$, $d\geqslant 1$, $E:=\overline{D}$,
$c$ is a bounded positive, real-valued Borel measurable function on $D$, extended with zero on $\mathbb{R}^d\setminus D$,
$\left( b_{k}\right) _{k\geqslant 1}$ is  a sequence of positive Borel measurable
functions on $E$ such that $\sum_{k\geqslant 1}b_k\leqslant 1$, 
for each $k \geqslant 1$  
$B_{k}$ is a Markovian kernel from $E%
^{\left( k\right) }$, 
 the $k$-th symmetric power of $E$, 
to $E$,  and 
$\varphi : \partial D \longrightarrow \mathbb{R}_{+}$ is  a bounded Borel measurable function.
The Laplace operator $\Delta$  is seen as the weak generator 
(in the sense of E.B. Dynkin)
of the Brownian motion on $\mathbb{R}^ d$ stopped at the entry time of the boundary of $D$,
$u$ is a positive real-valued Borel measurable function on $E$ 
which belongs to the domain of $\Delta$, 
such that $u\leqslant 1$, and for every $k\in\mathbb{N}$, $k\geqslant 1$, we denoted by $u^{\left( k\right) }
:E^{\left( k\right) }\longrightarrow \mathbb{R}$ the function defined as $u^{( k) }( 
\mathbf{x}) :=u( x_{1}) \cdot ...\cdot u(
x_{k}) $ for all $\mathbf{x}=( x_{1},...,x_{k}) \in 
E^{(k) }$.

Recall that a particular (actually, the classical) case of the problem (\ref{problem}) is:
\begin{equation} \label{prob1.2}
\left\{ 
\begin{array}{c}
(\Delta -c) u+c\underset{k\geqslant 1}{\sum }b_{k} u^{
k }=0\text{ in } D \\[4mm]
\underset{D \ni x \rightarrow y}{\lim} u(x)=\varphi(y) \text{ for all }y \in \partial D,%
\end{array}%
\right.  
\end{equation}%
obtained by taking the kernels $B_k$, $k\geqslant 1$, defined as 
$B_k g(x):=g( x,...,x) $ for  $x\in E$ and $g$ a Borel function on $E^{\left( k\right) }$; 
 see assertion $(iv)$ of Remark \ref{rem2.4} below.

The existence of a solution to the problem (\ref{problem}) was investigated  in 
\cite{LuSt17}; see also \cite{Hsu}. 
Notice that the boundary condition in (\ref{problem}) 
implies that $\varphi$ is a continuous function on $\partial D$, 
provided that the solution $u$ to  (\ref{problem}) is a continuous function on $D$.

Instead of $\mathbb{R}^d$ it is possible to consider
a  Lusin topological space $F$ (i.e., $F$ is homeomorphic to a Borel subset of a compact 
metrizable space), and $D$ a bounded domain of $F$ such that $E:=\overline{D}$ is a compact set.
The Laplace operator $\Delta$ is replaced in this general frame by the weak generator $L$ of a diffusion $X$ on $E$.
The corresponding version of the problem (\ref{problem}) (with $L$ instead of $\Delta$)  was investigated in 
\cite{BeOp11} for the case when $\varphi \leqslant 1$  
is a positive and continuous real-valued function on $\partial D$, 
generalizing a classical result of Nagasawa from \cite{Naga76}. 
In  \cite{BeOp14} it is treated the case when  $\varphi$ is a continuous bounded positive function.

The aim of this paper is twofolds.
First, we intend to solve the problem (\ref{problem}) for functions $\varphi$ which are discontinuous, replacing the pointwise convergence to the boundary data with an adequate convergence called "controlled convergence". 
The control is expressed with a $L$-superharmonic function on $D$,  
$k:D \longrightarrow \overline{\mathbb{R}}_+$, 
and we write $u\overset{k}{\longrightarrow } \varphi$ instead of the pointwise convergence of $u$ 
to the boundary data $\varphi$. 
So, for bounded and positive boundary data $\varphi$ on $\partial D$ we shall consider the boundary value problem 
\begin{equation} \label{problem2} 
	\left\{ 
	\begin{array}{l}
		( L -c) u+c\underset{k\geqslant 1}{\sum }b_{k}B_{k}u^{(
			k) }=0\text{ in }D \\ 
		u\overset{k}{\longrightarrow } \varphi %
	\end{array}%
	\right.    
\end{equation}%
for two cases:
\begin{enumerate}[]
	\item[(1.2i)] \label{L=BM} $L$ is the Laplace operator, more precisely, $L$ is the weak generator of the $d$-dimensional 
 Brownian motion,
 which is an extension of the Laplace operator. 
	\item[(1.2ii)] \label{L=Phi} $L$ is a gradient type operator, 
	the weak generator of a continuous flow $\phi=(\phi_t)_{t\geqslant 0}$ on $F$, leaving $D$ in finite time. 
 In this case, as in \cite{BeCoRo11}, 
 we have to consider an exceptional set for the controlled convergence to the boundary data in (\ref{problem2}). 
 Examples of such gradient type operators and flows are given 
 in Example \ref{exam3.3} below, in both finite and infinite dimensional frames.
\end{enumerate}

\vspace{-3mm}

A positive real-valued Borel measurable function $u$ on $E$ satisfying (\ref{problem2}) 
is called \textit{generalized solution} to the problem.

The second aim is to show that the problem  (\ref{problem2}) has
a {generalized solution} which admits a probabilistic representation.
As in \cite{BeLu16} and \cite{LuSt17}, a key tool of our approach is   
a non-local branching Markov process 
$\widehat{X}=(\widehat{X}_t, \widehat{\mathbb{P}}^{\mu},\mu\in \widehat{E})$ 
with state space the set $\widehat{E}$ of all finite configurations of $E$
(i.e. the set of all finite sums of Dirac measures concentrated at points of $E$ to which 
we add the zero measure $\mathbf{0}$ on $E$, 
endowed with the weak topology on the finite measures on $E$ 
and the corresponding Borel $\sigma $-algebra $\mathcal{B}(\widehat{E})$, 
that is, the set $\widehat{E}$ is identified with  
$\widehat{E}=\underset{k\geqslant 0}{\bigcup }E^{\left( k\right) }$,  where $E^{(0)}:=\{\mathbf{0} \}$);
$\widehat{X}$ has the spatial motion $X$ 
and the branching mechanism induced by the sequence of Markovian kernels 
$(B_k)_{k\geqslant 1}$ 
(for details see e.g. \cite{BeLu16} and \cite{BeLuVr20}). 

It is important to notice that in contrast with the linear Dirichlet problem on $D$ (the case $c=0$ in $(\ref{problem})$ and (\ref{problem2})),
for which the probabilistic representation of the solution is given by using a path continuous Markov process, describing the stochastic evolution in $D$ of a single particle (see also assertion (iii) of Remark \ref{rem2.4} below), in order to give a probabilistic representation for the nonlinear Dirichlet problem it is necessary a branching process, which describes the time evolution in $D$ of a system of particles.
We mention here the essential contribution of E.B. Dynkin in using the measure-valued superprocesses as
instruments for solving semilinear  equations,  
the typical one being  $\Delta u= u^\alpha$, with $1<\alpha\leqslant 2$; 
see the monographs \cite{Dynkin02}, \cite{Dynkin04},  and the references therein, see also \cite{Dawson}.

It turns out that in order to follow the above mentioned program initiated by E.B. Dynkin in  the early 1990s,  
for solving the equation  $( L -c) u+c {\sum }_{k\geqslant 1} b_{k}B_{k}u^{(k) }=0$ we need the  non-local branching process
$\widehat{X}$ instead of a superprocess.
Employing techniques from \cite{BeOp11}, \cite{BeOp14}, and \cite{BeLu16}  
(see also \cite{BeLuVr20}), we show that in the case of (1.2i), for some $r>0$ such that  $0\leqslant \varphi  < r$, 
there exists a solution $u$ to the nonlinear Dirichlet problem (\ref{problem2}) 
such that it admits a probabilistic representation, by using the measure-valued branching process $\widehat{X}$,	
$$
u(x)=\underset{t\rightarrow \infty }{\lim }r \widehat{\mathbb{E}}^{\delta _{x}}\{%
\widehat{\left( \frac{\varphi }{r}\right)  }(\widehat{X}_{t})\} \text{ for all }x\in E,   
$$
where $u$ converges to $\varphi$ controlled by a (real-valued) harmonic function $k$. 

In the case of (1.2ii), following \cite{BeVr22}, the measure-valued process $\widehat{X}$ admits a representation through 
the flow of measures induced by the spatial motion $\Phi$ 
(= the flow on $E$ obtained from $\phi$ by stopping it at the boundary of $D$) 
and 
a second branching Markov process 
$\widehat{X^0}= (\widehat{X^0_t}, \widehat{\mathbb{P}^0}^{\, \mu}, \mu \in \widehat{E})$ 
which has the same branching mechanism as $\widehat{X}$ but has 
no spatial motion; $\widehat{X^0}$  is called pure branching process. 
The probabilistic representation of the solution $u$ becomes 
$$
u(x)=\lim_{t\rightarrow \infty }r\widehat{%
	\mathbb{E}^{0}}^{\delta _{x}}
 \{\widehat{\left( \frac{\varphi }{r}\right) }%
(\Phi _{t}(\widehat{X^0_t} ))\}\text{ for all }x\in E,
$$ 
and $u$ converges to $\varphi$ controlled by a positive Borel measurable function $k$, outside an exceptional set.

The paper has the following structure. 
In Section \ref{sect2} we present preliminary results on the Markov processes: 
the entry and hitting times, regular boundary points, the transition function. 
In particular, we show  that the strong Feller property of a transition function is preserved by killing.
Further, it is given a version of the stochastic solution of the Dirichlet problem, 
then it is introduced the controlled convergence, and finally, 
{several results on the non-local branching processes are exposed}.
In particular, it is pointed out the
probabilistic meaning of the kernels
$B_k$.
Section \ref{sect3} is reserved to the main results on the nonlinear Dirichlet problem 
and the probabilistic representation of the solution.
In Subsection \ref{subsect3.1} we treat the case of the Laplace operator 
(Theorem \ref{thm1.1}), while in Subsection \ref{subsect3.2} 
we investigate the case of the gradient type operator (Theorem \ref{thm3.3}).
Assertion $(iv)$ from Remark \ref{rem2.4} is a result 
on the smallness of the set  on the boundary where the solution does not converge pointwise to the given boundary data.
The last section is an Appendix, collecting  some facts
on the fine topology, the weak generator, and the non-local branching processes;
we also put here the proofs of several results from Section \ref{sect2}.

\section{Preliminary results, controlled convergence,
and  \linebreak 
branching processes} \label{sect2}

Let $\mathcal{B}(F)$
be the Borel $\sigma$-algebra of the Lusin topological space $F$.
For  $A\in \mathcal{B}(F)$, we denote by $\mathcal{B}(A)$ the Borel $\sigma$--algebra of $A$, $\mathcal{B}(A)=
\mathcal{B}(F)|_A$.
Let further $\mathcal{B}_+(A)$ denote the convex cone of 
all numerical, positive $\mathcal{B}(A)$-measurable functions on $A$ and $b\mathcal{B}_+(A) 
$ be the set of bounded functions from  $\mathcal{B}_+(A)$.

Let $Y=(Y_t, \mathbb{P}^{x}, x\in F )$ be a diffusion on $F$, that is, 
a path continuous right Markov process with state space $F$.

Let 
$\tau$ be the {\it first entry time}  of $\partial D$,
$\tau:=\inf \{t\geqslant 0: Y_{t} \in \partial D\}$,  and 
$\tau_o$ be the {\it first hitting time}  of $\partial D$, $\tau_o:= \inf\{ t>0 : Y_t\in \partial D\}$. 
Recall that if $x\in D$ then
$\tau=\tau_o$ $\mathbb{P}^{x}$-a.s. 

Define the kernel $P_{\partial D}$ on $F$ as follows:
$$ 
P_{\partial D}f (x) = \mathbb{E}^x \{ f(Y_{\tau_o}); \tau_o <\infty\},\  f\in \mathcal{B}_+(F),\  x\in F.
$$
It is called the {\it hitting kernel of}
$\partial D$; we have denoted by $\mathbb{E}^{x}$ the expectation under $\mathbb{P}^{x}$. 

{We consider the process $Y$ stopped at the boundary of $D$, that is, 
the process $X=(X_t, \mathbb{P}^x , x\in E)$ with state space $E$, defined as $X_{t}:=Y_{t\wedge \tau }$.
Recall that we denoted by $E$  the closure of $D$.
Suppose  
that $\mathbb{P}^{x}( \tau <\infty
) =1$ for all $x\in D$. Let $L$ be the weak generator associated with the process $X$; see the Appendix below.}

Let $(T_t)_{t\geqslant 0}$ be the transition function of $X$, that is, 
for every $f\in b\mathcal{B}_+(E)$, $T_tf(x):=\mathbb{E}^xf(X_t)$, $t\geqslant 0$.  
 Let further $( T^c_{t}) _{t\geqslant 0}$ be the
transition function of the process obtained from $X$ by killing with the
multiplicative functional induced by $c,$ expressed as  the {\it Feynman-Kac
semigroup}, 
\begin{equation*}
	T^c_{t}f\left( x\right) =\mathbb{E}^{x}\{ e^{-\overset{t}{\underset{0}{\int }}%
		c( X_{s}) ds}f( X_{t})\}  ,\text{ }f\in b\mathcal{B}_+(E), x\in E.
\end{equation*}%
Its corresponding  weak generator is $L -c$, 
where $L$ is the weak generator of $(T_t)_{t\geqslant 0}$,  and we have $\mathcal{D}(L-c)=\mathcal{D}(L)$.

Define  the kernel $P_{\tau}^{c}$ on $F$ as 
$$
P_{\tau }^{c}f( x) := 
\mathbb E^{x}\{e^{-\int_{0}^{\tau }
c(Y_{s}) ds}f\left( Y_{\tau }\right)
	\}, \ 
 f\in b\mathcal{B}_+(F), x\in F.
$$
Defining analogously the kernel $P_{\tau_o }^{c}$,  we have 
$P_{\tau }^{c}f( x)=P_{\tau_o }^{c}f( x)$ if $x\in D$.

We claim that 
\begin{equation} \label{eq2.1}
	\underset{t\rightarrow \infty }{\lim }T^c_{t}(f|_E)\left( x\right)
 =P_{\tau }^{c}f( x) \text{, }x\in E, f\in b\mathcal{B}_+(F).
\end{equation}
The proof of $(\ref{eq2.1})$ is given in the Appendix.

\vspace{2mm}

\noindent{\bf Strong Feller transition function}. 
A transition function $(T_t)_{t\geqslant 0}$ on a Lusin topological space $F$ is called 
{\it strong Feller} provided that $T_tf$ is a continuous function on $F$ for every $f\in b\mathcal{B}_+(F)$ and $t>0$. 

The next result shows that the strong Feller property of a transition function is preserved by killing. 
We present its proof in the Appendix.  
A related result, under restrictive conditions, was proved in \cite{Chung86}, Theorem 2.

\begin{lem} \label{lemma2.1}  
Let $(S_t)_{t\geqslant 0}$ be the transition function of a right Markov process $Y$ on the Lusin topological space $F$, 
$c\in b\mathcal{B}_+(F)$, and consider 
$( S^c_{t}) _{t\geqslant 0}$, the
transition function of the process obtained from $Y$ by killing with the multiplicative functional induced by $c$. 
Then the following assertions hold.

$(i)$ If $(S_t)_{t\geqslant 0}$  is strong Feller
then so is $( S^c_{t}) _{t\geqslant 0}$.

$(ii)$ Assume in addition that $F$ is a locally compact space with a countable base. 
If  $(S_t)_{t\geqslant 0}$ acts on the space of continuous functions vanishing at infinity then 
$( S^c_{t}) _{t\geqslant 0}$ has the same property. 
\end{lem}

Recall  that a point $x\in \partial D$ is said to be {\it regular boundary point} of $D$ if  
$\mathbb{P}^{x}(\tau_o =0)=1 $; see e.g. \cite{Dynkin}, vol. II, page 32, or
\cite{ChungZhao95}, page 23.
The domain $D$ is called {\it regular} provided that every point  of $\partial D$ is a regular boundary point of $D$. 

\vspace{1mm}

Denote by $V$ the kernel on $F$  defined as
$$
Vf (x)= 
\mathbb{E}^{x} \!\!
\int_0^\infty 
e^{-\overset{t}{\underset{0}{\int }} c( Y_{s}) ds}
f( Y_{t}) dt, \ 
f\in b\mathcal{B}_+(F), x\in F,
$$
it is the {\it potential kernel} of the process on $F$ obtained from Y
by killing with the multiplicative functional induced by $c$. \\

\noindent
{\bf The stochastic solution to the Dirichlet problem.}
The following result is a version of 
Theorem 13.1 from \cite{Dynkin}, vol II, page 32. Its proof is based essentially on the above mentioned
result from \cite{Dynkin}, but uses also
Lemma \ref{lemma2.1}, and several arguments from the potential theory  for Markov processes.
We put it in the Appendix.

\begin{prop} \label{prop2.2}
Assume that $F$ is a locally compact space with countable base.
If $Y$ 
is a right Markov process with state space $F$, having continuous paths, 
such that its transition function $(S_t)_{t\geqslant 0}$ is strong Feller and acts on continuous functions vanishing at a infinity, and
$D$ is a regular domain with compact closure. 
Suppose in addition that the potential kernel $V$  is {\rm proper}, that is, 
there exists a function $h\in \mathcal{B}_+(F)$, $h>0$, such that $Vh$ is a real-valued function.
Then  
\begin{equation} \label{continuous}
	\underset{D\ni x\rightarrow y}{\lim }P_{\tau }^{c}f(x) =f(y) \text{ for every } 
 f\in b\mathcal{C}_+\left( \partial D\right)
 \text{ and } 
 y\in \partial D. 
\end{equation}
\end{prop}

\noindent
{\bf Controlled convergence}; cf. \cite{Cornea95} and \cite{Cornea98}.
Let $f:\partial D\rightarrow \overline{\mathbb{R}},$ $%
D_{o}\subset D$ and $h,k:D\longrightarrow \overline{\mathbb{R}},$ $%
k\geqslant 0,$ such that $h|_{D_{o}},$ $k|_{D_{o}}$ are real-valued. We say that $%
h$ converges to $f$ controlled by $k$ on $D_{o}$ (we write $h\overset{k}{\longrightarrow }f$ on $D_o$) if for every $A\subset
D_{o} $ and $y\in \partial D\cap \overline{A}$ the following conditions hold:

\begin{enumerate}
	\item[$( \ast ) $] If $\underset{A\ni x\rightarrow y}{\lim \sup }$
	$k( x) <\infty $ then $f( y) \in \mathbb{R}$ and $f(
	y) =\underset{A\ni x\rightarrow y}{\lim }h(x), $
	
	\item[$( \ast \ast ) $] If $\underset{A\ni x\rightarrow y}{\lim }$
	$k( x) =\infty $ then $\underset{A\ni x\rightarrow y}{\lim }\frac{%
		h( x) }{1+k( x) }=0.$
\end{enumerate}
 If the set $D_o$ is not specified, then $D_o=D$ and we write $%
h$ converges to $f$ controlled by $k$. This case was considered in \cite{Cornea98} and \cite{BeTeodor}.

If $ g: \partial D \longrightarrow \overline{\mathbb{R}}$ is a  bounded below Borel measurable function and $x\in D$ define
$$
H_D g(x) := \mathbb{E}^{x} g (Y_{\tau}).
$$

We state now a result on the controlled convergence for the stochastic solution of the linear Dirichlet problem
with discontinuous boundary data.
It turns out that we can exploit it in solving the nonlinear boundary value problem 
(\ref{problem2}).
In the linear case this reasoning was first used in the proof of 
Theorem 5.3 from  \cite{BeCoRo11}; see also
Theorem 4.8 from \cite{Be11},  and Theorem 2.5 from \cite{BeTeodor}.
For the reader's convenience we present a sketch of the  proof in the Appendix.

\begin{lem} \label{Lema} 
Let $\lambda$ be a finite measure on $D$ and $\sigma:=\lambda \circ P^c_\tau$. 
Assume that (\ref{continuous}) holds, and let $\varphi\in L_+^1(\partial D, \sigma)$, 
then there exists a Borel measurable function 
$g:\partial D \longrightarrow \overline{\mathbb{R}}_+$ 
such that $P_{\tau }^{c}\varphi \overset{k}{\longrightarrow }\varphi$ on $D_o=[k<\infty]$, 
where $k:= P^c_\tau g\in L^1(D,\lambda)$.
    \end{lem}

\begin{rem} \label{remarca}
$(i)$ 
The set $D\setminus D_o = [k=\infty]$ 
in  Lemma \ref{Lema} is finely closed, $\lambda$--polar and $\lambda$-negligible.
    
$(ii)$  In Theorem 2.5 from \cite{BeTeodor} it was considered the case when   
    $F=\mathbb{R}^d$, $d\geq 1$,   $Y$ is the $d$-dimensional Brownian motion on $\mathbb{R}^d$,
    $c\equiv 0$, 
    $\varphi : \partial D \longrightarrow \mathbb{R}$ is a bounded below Borel measurable function, such that there exists $x\in D$ with $H_D\varphi(x)<\infty$, and  $\sigma:=\delta _{x}\circ H_{D}$. 
    Here, $k$ is a real-valued harmonic function on $D$ and $D_o=D$.  
\end{rem}

 Let  $\mathcal{B}_{1}$ be the set of all functions $f\in \mathcal{B}_+(E)$
 such that $f\leqslant 1$. Assume that 
 \begin{equation}
 	\underset{x\in E} {%
 		\sup }\underset{k\geqslant 1} {\ \sum }kb_{k}\left( x\right) <\infty .
 	\label{conditions}
 \end{equation}%

\noindent
{\bf Non-local branching processes.}
Recall that a right Markov process with state space $\widehat{E}$ is called 
\textit{branching process} provided that for any independent copies $X_1$ and $%
X_2$ of the given process on $\widehat{E},$ starting respectively
from two measures $\mu_1 $ and $\mu_2$  from $\widehat{E}$, $X_1+X_2$ and the
process starting from $\mu_1 +\mu_2$ are equal in distribution (for more details see Appendix below).

The following assertions hold.
\begin{enumerate}
	\item[(A)] For every $f \in \mathcal{B}_1$, the equation 
	\begin{equation}
		w_{t}(x) =T^c_{t}f(x) +\underset{0}{\overset{t}{%
				\int }}T^c_{s}(c\underset{k\geqslant 1}{\sum }b_{k}B_{k}w_{t-s}^{(k)})(x) ds,\text{ }t\geqslant 0,x\in E,   \label{Ht}
	\end{equation}%
	has a unique solution $(t,x) \longmapsto H_{t}f (x) $ jointly measurable in $(t,x)  $ such that $%
	H_{t}f \in \mathcal{B}_1$,  for every $t\geqslant 0$, and the family $( H_{t})
	_{t\geqslant 0}$ is a nonlinear semigroup of operators on $\mathcal{B}_{1}$ 
	(cf.  \cite{BeJEMS}, Proposition 3.2, and \cite{BeLu16}, Proposition 4.1; see also
	\cite{BeOp11}, Proposition 6.2).
	
	\item[(B)] There exists a sub-Markovian semigroup $(\widehat{%
		\mathbf{H}}_{t})_{t\geqslant 0}$ on $(\widehat{E},$ $\mathcal{B}(\widehat{E}))$, formed from branching kernels, 
	such that $\widehat{\mathbf{H}}_{t}\widehat{f }=\widehat{H_{t}f }
	$ for all $f \in \mathcal{B}_{1}$, $t\geqslant 0$, and if $u\in \mathcal{B}_{1}$ and $%
	\widehat{u}$ is an invariant function with respect to this semigroup, then $u
	$ belongs to the domain of the weak generator 
	$L-c$ and satisfies the following nonlinear equation 
	\begin{equation*}
		(L-c)u+c\underset{k\geqslant 1}{\sum }b_{k}B_{k}u^{(k) }=0  
	\end{equation*}%
	(cf. \cite{BeOp11}, Proposition 6.1).
 	
	\item[(C)] Let 
	 $\widehat{X}=(\widehat{X}_{t},\widehat{\mathbb{P}}%
	^{\mu},\mu\in \widehat{E})$ be the branching Markov process with state space $\widehat{E}$, which has $(\widehat{\mathbf{H}}%
	_{t})_{t\geqslant 0}$ as its transition function. 
	Consequently, if $f \in \mathcal{B}_1$ then the solution $(H_t f)_{t\geqslant 0}$ of the equation (\ref{Ht}) has the
	representation%
	\begin{equation}
		H_{t}f \left( x\right) =
  \widehat{\mathbf{H}}_{t}\widehat{f }%
		\left( \delta _{x}\right) =
  \widehat{\mathbb{E}}^{\delta _{x}}
  \widehat{f } (\widehat{X}_{t}),\text{ }x\in E,\text{ }t\geqslant 0,
		\label{probabilistic}
	\end{equation}%
cf. \cite{BeLu16}, Theorem 4.10; see also \cite{BeLuVr20}, Theorem 4.1.
\end{enumerate}

We close this section with the probabilistic meaning 
of the branching mechanism for the branching process $\widehat{X}$: if $k > 1$ and $x\in E$, then $b_k(x)$
is the probability that a particle destroyed at the point $x$ has $k$ descendants, 
while the probability $B_{k,x}$ induced by the Markovian kernel $B_k$ is
the distribution of the $k$ descendants; recall that
$B_{k,x}$ is the probability measure on $E^{(k)}$
such that $\int_{E^{(k)}} f d B_{k,x}
= B_{k}f(x)$ for every $g\in \mathcal{B}_+(E^{(k)})$. 
Consequently, such a branching process is named {\it non-local}, 
since  the descendants are not born  from the point where the parent died.
Only in the  particular case related to the problem (\ref{prob1.2}) the branching process 
has the property that the descendants always 
start from the point where the parent died, because in this case we have 
$B_{k,x}=\delta_x$ for all $x\in E$.

\section{Dirichlet problem of non-local branching processes} \label{sect3}

\subsection{The case of the Laplace operator} \label{subsect3.1}

Let $F=\mathbb{R}^d$, $d\geqslant 1$, $D\subset \mathbb{R}^d$ 
a bounded regular domain, and $E=\overline{D}$. 
Let $Y=(Y_{t},\mathbb{P}^{x}, x\in \mathbb{R}^d)$ 
be the $d$-dimensional Brownian motion on $\mathbb{R}^d$, $\tau$  the first entry time of $\partial D$ of the Brownian motion.

If $ g: \partial D \longrightarrow \overline{\mathbb{R}}$ 
is a  bounded below Borel measurable function, by Theorem 3.7 from \cite{Port-Stone}   
it follows that $H_D g$ is a (real-valued) harmonic function on $D$ provided that $H_D g$ is not equal $+\infty$ everywhere on $D$.

Let $X=(X_t)_{t\geqslant 0}$ 
be the stopped Brownian process at the boundary of $D$. Recall that $X$ is a conservative path continuous Markov process. 
Let $L=\Delta$, i.e., $L$ is the 
weak generator of the d-dimensional Brownian motion stopped at the boundary of $D$.

If $ f: E \longrightarrow \mathbb{R}$ 
we denote by $\widehat{f}: \widehat{E} \longrightarrow \mathbb{R}$ the function defined as 
$\widehat{f} |_{E^{(k)}}=f^{(k)}$ if $k \geqslant 1$ and $\widehat{f} (\mathbf{0})=1$.

We can state now our first main result.

\begin{thm} \label{thm1.1}  
	\label{main} Let $\varphi : \partial D \longrightarrow \mathbb{R}_{+}$ 
	be a bounded Borel measurable function, extended as a function on $E$ with the value zero on $D$.
 Let $r>0$ be such that $\|\varphi\|_{\infty}\leqslant r$. Assume that 
	\begin{equation}
\sup_{x\in E}	
 \underset{k\geqslant 1}{\sum }%
		r^{k-1}b_{k}\left( x\right) \leqslant 1
  \mbox{ and } \, 
			\sup_{x\in E} \sum_{k\geqslant 1} kr^{k-1}b_{k}\left( x\right) <\infty .
		\label{conditions}
	\end{equation}%
	Then there exists a constant $c_{0}>0$ such that if $c<c_{0}$, 
	then there exist a Borel measurable function  $g: \partial D \longrightarrow \overline{\mathbb{R}}_{+}$ and a non-local
	branching process $\widehat{X}=(\widehat{X}_{t}, \widehat{\mathbb{P}}^{\mu}, \mu\in \widehat{E})$
	with state space $\widehat{E}$ (the set of all finite
	configurations of $E$), and spatial motion  $X$, such that there exists
	\begin{equation}
		\underset{t\rightarrow \infty }{\lim }r\widehat{\mathbb{E}}^{\delta _{x}}\{%
		\widehat{( \frac{\varphi }{r}) }(\widehat{X}_{t})\} 
  =:u\left(
		x\right) \text{ for all }x\in E,  \label{existalimita}
	\end{equation}%
	and $u$ is a generalized solution to the nonlinear 
	Dirichlet problem $(\ref{problem2})$ with boundary data $\varphi $, associated with the operator $%
	u\longmapsto \left( \Delta -c\right) u+c\underset{k\geq 1}{\sum }%
	b_{k}B_{k}u^{\left( k\right) }$, that is 
	\begin{equation} 
		\label{Delta}
		\left\{ 
		\begin{array}{l}
			\left( \Delta -c\right) u+c\underset{k\geq 1}{\sum }b_{k}B_{k}u^{\left(
				k\right) }=0\text{ in }D \\ 
				u\overset{k}{\longrightarrow } \varphi ,%
		\end{array}%
		\right.
	\end{equation}%
	where $k$ is a 
 harmonic function on $D$,  defined as  $k=H_D g$.
\end{thm}

\begin{rem} \label{rem2.4} 
	(i)
	The expression given by $(\ref{existalimita})$ 
 represents the 
	\textbf{stochastic solution to the nonlinear Dirichlet problem} $(\ref%
	{problem})$  with bounded  boundary data $\varphi $.
	
	(ii)
	It was proved in \cite{LuSt17} that if in addition the boundary data $\varphi $ is
	continuous on $\partial D$ then the problem  (\ref{Delta}) 
	can be solved such that the limit of the solution $u$ to $\varphi$ holds pointwise.
	
	(iii)
	The case of the linear Dirichlet problem with discontinuous, bounded boundary value problem was treated  in \cite{Be11} and \cite{BeTeodor}, where
	it was proven that the stochastic solution (induced by the Brownian motion on $D$)
	solves the linear Dirichlet problem in the sense of the controlled convergence.
	In \cite{Lukes10} it is given a different approach, 
	by using the Perron-Wiener-Brelot method instead of the stochastic solution. 
 A nonlinear Dirichlet problem with discontinuous boundary data was solved in \cite{BeTeodor}.

 (iv) Recall that a subset $M$ of $\partial D$ is called {\rm negligible} provided that $H_D (1_M)\equiv 0$ on $D$. 
 According to Theorem 6.5.2 from \cite{ArGar01} and Proposition 2.7 from \cite{Cornea98} 
 it follows that the set on the boundary where the solution $u$ 
 does not converge pointwise to the given boundary data $\varphi$ is negligible. 
 Notice that the property of the control function $k$ to be superharmonic on $D$ is used here in an essential way.

 (v) For each $k\geqslant  1$ define the kernel $B_{k}$ as $B_{k}g\left( x\right)
	:=g\left( x,...,x\right) $ for all $x\in E$ and $g\in b\mathcal{B%
	}_+(E^{\left( k\right) }).$ 
	Then $B_{k}u^{\left( k\right) }=u^{k}$
	for all $u\in \mathcal{B}_+(E)$, $u\leqslant 1$, 
	and  $B_{k}$ is a Markovian kernel for all $k\geqslant 1.$ 
	Equation (\ref{Ht})
	becomes 
 \begin{equation*}
		w_{t}\left( x\right) =T_{t}\varphi \left( x\right) +\underset{0}{\overset{t}{%
				\int }}T_{s}(c\underset{k\geqslant 1}{\sum }b_{k}w_{t-s}^{k})\left( x\right) ds,%
		\text{ }t\geqslant 0,\text{ }x\in E,
	\end{equation*}
 while the problem (\ref{problem2}) reduces to 
 \begin{equation} 
	\left\{ 
	\begin{array}{l}
		( L -c) u+c\underset{k\geq 1}{\sum }b_{k}u^{
			k }=0\text{ in }D \\ 
		u\overset{k}{\longrightarrow } \varphi. %
	\end{array}%
	\right.    
\end{equation}%
We already mentioned in the Introduction
that with these kernels the problem
(\ref{problem}) is transformed in 
(\ref{prob1.2}).
\end{rem}

\begin{proof}[Proof of Theorem \ref{thm1.1}]
\noindent
Assume that $r=1$.
From conditions (%
\ref{conditions}) 
and (A), the 
equation (\ref{Ht}) has a unique solution $( t,x) \longmapsto $ $%
H_{t}f (x)$, $t\geqslant 0.$ Applying (B) and (C), there exist a sub-Markovian branching semigroup $(\widehat{\mathbf{H}}%
_{t})_{t\geqslant 0}$ such that $\widehat{\mathbf{H}}_{t}\widehat{f }=%
\widehat{H_{t}f },$ and a branching Markov process $\widehat{X}=(\widehat{X_{t}}, \widehat{\mathbb{P}}^\mu, \mu \in \widehat{E})$ on $\widehat{E}$ which has $(\widehat{\mathbf{H}}%
_{t})_{t\geqslant 0}$ as transition function. Hence $\widehat{H_{t}f }%
( \mu ) =\widehat{\mathbb{E}}^{\mu }\widehat{f}(\widehat{X%
}_{t})$ for all $\mu \in \widehat{E},$ $t\geqslant 0,$ and $f
\in \mathcal{B}_+(E),$ $f \leqslant 1,$ and 
(\ref{probabilistic}) holds.

Arguing as in the proof of Theorem 2.1 from \cite{LuSt17}, we have that condition (6) from Nagasawa's
paper \cite{Naga76} is verified, more precisely, 
there exists a constant $c_{0}$ such that if $c<c_{0}$, then for some $%
\varepsilon >0$%
\begin{equation*}
\mathbb{P}^x(\tau <R)\geqslant \varepsilon \text{ for all }x\in E,
\end{equation*}%
where $R$ is the killing time of  $X$  by the 
multiplicative functional induced by $c$.
It follows from  Proposition 2 in \cite%
{Naga76} that for every $x\in E$ and any function $\varphi \in \mathcal{B}_+(\partial D)$ with $\varphi \leqslant 1$,  
$u( x ) :=\lim_{t\to \infty }\widehat{%
\mathbb{E}}^{\delta _{x}}\widehat{\varphi }(\widehat{X}_{t})$ exists, even if  $\varphi$ is not necessarily continuous; 
recall that the boundary data $\varphi$ is considered here as a function on $E$
vanishing on $D$.

We show that $u$ is a generalized solution to the nonlinear Dirichlet
problem (\ref{Delta}) with boundary data $\varphi .$
As in the proof of Proposition 6.4 from \cite{BeOp11}, the
function $\widehat{u}$ is invariant with respect to $(\widehat{\mathbf{H}}%
_{t})_{t\geqslant 0.}$ So, from (B), it follows that $u$ belongs to the domain of $%
\Delta $ and $\left( \Delta -c\right) u+c\underset{k\geqslant 1}{\sum }%
b_{k}B_{k}u^{\left( k\right) }=0$ holds on $D.$

Now we prove that $u$ converges to $\varphi$ in the sense of the controlled convergence. 
Letting $t\rightarrow \infty $ in $H_{t}\varphi =T^c_{t}\varphi
+\int_{0}^{t}T^c_{s}(c\underset{k\geqslant 1}{\sum }b_{k}B_{k}\left( H_{t-s}\varphi
\right) ^{\left( k\right) })ds$ we get that $u=P_{\tau }^{c}\varphi +v,$
where $v:=\underset{t\rightarrow \infty }{\lim }\int_{0}^{t}T^c_{s}(c\underset{%
k\geqslant 1}{\sum }b_{k}B_{k}\left( H_{t-s}\varphi \right) ^{\left( k\right)
})ds.$ As in the proof of Therorem 2.1 from \cite{LuSt17} 
(see also the proof of Proposition 6.4 from \cite{BeOp11}), we have that $\underset{D\ni
x\rightarrow y}{\lim }v\left( x\right) =0$ for all $y\in \partial D.$

Since $\varphi $ is not necessarily
continuous on $\partial D$, we cannot apply Proposition \ref{prop2.2}, 
however, we can apply Lemma \ref{Lema}  with $\lambda$ the Lebesgue measure on $D$ 
and we get that there exists $g \in \mathcal{B}_+(\partial D)$ such that 
$P_{\tau }^{c}\varphi $  converges to $\varphi $ controlled by 
$P_{\tau}^cg\in L^1(D, \lambda)$ on $D_o :=[P_{\tau}^cg < \infty]$. 
We claim that $D_o=D$.
Indeed,  by Theorem 4.7 (i) from \cite{ChungZhao95}, applied for the Kato class function $q=-c$, we
have
$$
P^c_\tau g + G(cP^c_\tau g) =
H_D g \ \mbox{ on D}
\mbox{ for every } g\in b\mathcal{B}_+(\partial D),
$$
where $G$ is the Green operator for the domain $D$.
The above equality extends to every Borel measurable positive function $g$ on $\partial D$ and 
by Lemma 2.11 from \cite{ChungZhao95} we have 
$G(L^1(D, \lambda))\subset L^1(D, \lambda)$. 
It follows that $H_D g$ also belongs to
$L^1 (D, \lambda)$. 
We argue now as in \cite{BeTeodor}.
Because  $H_D g$ is not equal $+\infty$ everywhere on $D$, by
Theorem 3.7 from \cite{Port-Stone},  page 106, it follows that 
$H_D g$ is a (real-valued) harmonic
function on $D$, hence $D_o=D$ as claimed,  because 
$P^c_{\tau} g \leqslant H_D g$. 
It follows that $k:= H_D g$ is also a real-valued control function and  $u\overset{k}{\longrightarrow }\varphi$. 

Suppose now that $r>1.$ 
Arguing as in \cite{BeOp14}, let $\psi :=\frac{\varphi 
}{r},$ 
then clearly $\left\vert \left\vert \psi \right\vert \right\vert _{\infty
}\leqslant 1$. 
Let 
\begin{equation*}
b_{k,r}:=r^{k-1}b_{k},\text{ }k\geqslant 1.
\end{equation*}%
Conditions (\ref{conditions}) become $\underset{k\geqslant 1}{\sum }b_{k,r}\left(
x\right) \leqslant 1$ and $\underset{x\in E}{\sup }\ \underset{%
k\geqslant 1}{\sum }kb_{k,r} \left( x\right) <\infty $. 
Therefore, applying (A), (B) and (C) for the sequences $\left(
b_{k,r}\right) _{k\geqslant 1}$ and $\left( B_{k}\right) _{k\geqslant 1}$, and the
measurable and subunitary boundary data $\psi$, there exist a semigroup $%
\left( H_{t}^{r}\right) _{t\geqslant 0}$ of nonlinear operators on $\mathcal{B}%
_{1}$, a {sub-Markovian} semigroup $(\widehat{\mathbf{H}^{r}_{t}})_{t\geqslant 0}$
of branching kernels on $(\widehat{E},$ $\mathcal{B(}\widehat{%
E})),$ and a Markov branching process $\widehat{X^{r}}=(\widehat{%
X^{r}_{t}},\widehat{\mathbb{P}^{r} }^{\mu},\mu\in \widehat{E})$ which
has $(\widehat{\mathbf{H}^{r}_{t}})_{t\geqslant 0}$ as transition function,
such that $H_{t}^{r}\psi \left( x\right) =\widehat{\mathbf{H}^{r}_{t}}%
\widehat{\psi }\left( \delta _{x}\right) =\widehat{\mathbb{E}^{r}}^{\delta
_{x}} 
\widehat{\psi }(\widehat{X^{r}_{t}})$ for all $x\in E$ and 
$t\geqslant 0.$

With the above considerations, applying the arguments from before to $%
\left( b_{k,r}\right) _{k\geqslant 1}$, $\left( B_{k}\right) _{k\geqslant 1}$ and $%
\psi $, we have that there exist a constant $c_{0}>0$ and a function $g\in \mathcal{B}_+%
( \partial D ) $ such that if $c<c_{0}$ 
\begin{equation*}
v(x):=\underset{t\to \infty }{\lim }\widehat{\mathbb{E}^{r}}^{\delta _{x}}%
\widehat{\psi }(\widehat{X^{r}_{t}}) \text{ for all }x\in 
E
\end{equation*}%
 is the generalized solution of the nonlinear Dirichlet problem with
boundary data $\psi$, 
associated with the operator $v\longmapsto $ $%
\left(\Delta -c\right) v+c\underset{k\geqslant 1}{\sum }b_{k,r}B_{k}v^{\left(
k\right) }$, that is%
\begin{equation*}
\left\{ 
\begin{array}{l}
\left( \Delta -c\right) v+c\underset{k\geqslant 1}{\sum }b_{k,r}B_{k}v^{\left(
k\right) }=0\text{ in }D \\ 
	v\overset{k}{\longrightarrow } \psi ,%
\end{array}%
\right.
\end{equation*}%
where $k:=H_{D}g$ is a harmonic function.
Moreover, taking $u:=r v,$ we have 
$$
u\left( x\right) =\underset{t\rightarrow \infty }{\lim }rH_{t}^{r}\psi
\left( x\right) =\underset{t\rightarrow \infty }{\lim }r\widehat{\mathbf{H}%
^{r}_{t}}\widehat{\psi }\left( \delta _{x}\right)
=\underset{t\rightarrow \infty }{\lim }r\widehat{\mathbb{E}^{r}}^{\delta
_{x}}\widehat{\psi }(\widehat{X_{t}^{r}}),\text{ }x\in E,\text{ }t\geqslant 0.
$$
Since $\left( \Delta -c\right) v+c\underset{k\geqslant 1}{\sum }%
b_{k,r}B_{k}v^{\left( k\right) }=0$ on $D,$ we obtain that $%
\left( \Delta -c\right) u+c\underset{k\geqslant 1}{\sum }b_{k}B_{k}u^{\left(
k\right) }=0$ on $D$ and because $v\overset{k}{\longrightarrow }\psi $, 
we get that $u\overset{k}{\longrightarrow }\varphi $, completing the proof.
\end{proof}

\subsection{The case of the gradient type operator} \label{subsect3.2}

Let $F$ be a Lusin topological space, $D$ a bounded domain of $F$, 
$E=\overline{D}$, and $\phi =(\phi _{t})_{t\geqslant 0}$ be a \textit{right continuous flow} on $F$.
 \noindent
  Recall that a right continuous flow on $F$ is a family $\phi=(\phi_t)_{t\geqslant 0}$ 
  of mappings on $F$  (cf. \cite{Sharpe}, page 41; see also \cite{BeIoLu}) provided that:
 
 (1) $\phi_{t+s}(x)=\phi_t(\phi_s(x))$ for all $s,t>0$ and $x\in F$;
 
 (2) $\phi_0(x)=x$ for all $x\in F$;
 
 (3) For each $t>0$ the function $F\ni x\longmapsto \phi_t(x)$ is $\mathcal{B}(F)/\mathcal{B}(F)$-measurable;
 
 (4) For each $x\in F$ the function $t\longmapsto \phi_t(x)$ is right continuous on $[0,\infty)$.

 \vspace{2mm}

 The right continuous flow $\phi $ may be regarded as a deterministic right Markov process on $F$
 with infinite lifetime, $Y=(\Omega ,\mathcal{F},\mathcal{F}_{t},{Y}_{t},\theta
 _{t}, {\mathbb{P}}^{x})$: $\Omega =F,$ $\mathcal{F}=\mathcal{F}_{t}=\mathcal{B}(\mathbb{R}^d),$ ${Y}_{t}(x):=\phi _{t}(x)=:\theta
 _{t}(x)$ for all $x\in \Omega ,$ and $ {\mathbb{P}}^{x}=\delta _{x}$. 
 Let $(S_t)_{t\geqslant 0}$ be the transition function of $Y$ (that is, of $\phi$),  
 $S_tf(x)=f(\phi_t(x))$ for all $t\geqslant 0$, $x\in F$, and $f\in \mathcal{B}_+(F)$. In particular,  the transition function 
 $(S_t)_{t\geqslant 0}$ on $F$ is Markovian. 
 
 Let $\Lambda$ be the weak generator of $Y$. 
 We say also that $\Lambda$ is the {\it weak generator of $\phi$}.
 It  is a first order  "gradient type operator"  
 in the sense that the domain $\mathcal{D}(\Lambda)$ of $\Lambda$ 
 is an algebra and if $u\in \mathcal{D}(\Lambda)$ then $\Lambda(u^2)=2u\Lambda u$; cf. \cite{BeBezzCi24}.
\vspace{1mm}

\begin{exam} \label{exam3.3}
 $(i)$ The classical example is the Euclidean gradient flow, 
 which was used e.g. in \cite{BeIoLu} to describe a multiple-fragmentation process
 to model the time evolution of a system of particles which move on an Euclidean surface,  driven by a given force and split in smaller fragments. 
 More precisely, we fix a smooth  Euclidean vector field $\mathbf{B}\in bC^ 1(\mathbb{R}^m, \mathbb{R}^m)$ 
 and consider  the continuous flow on $\mathbb{R}^m$ having the (weak) generator
 $\Lambda u = \mathbf{B}\cdot \nabla u$, $u\in bC^ 1(\mathbb{R}^m)$. 
 
 $(ii)$ However, it is possible to have relevant examples of  continuous flows 
 and gradient type operators in  an infinite dimensional frame. 
 An example is the measure-valued branching continuous flow studied in \cite{BaBe16}.
 Namely, let $(T_t)_{t\geqslant 0}$ be a transition function on a Lusin topological space $F$. 
 We assume that $(T_t)_{t\geqslant 0}$  is a
 {\it Feller transition function}, that is,  $T_tf$ is a continuous function on $F$ for every $f\in bC(F)$ and $t>0$.
 Actually, in \cite{BaBe16} it was considered the particular case of the transition function 
 of the reflecting Brownian motion on the closure of a bounded Euclidean domain.
 Let $M(F)$ denote the convex cone of all positive finite measures on $F$.
 Endowed with the weak topology $M(F)$ is also a Lusin topological space.
 For each $t\geqslant 0$ and $\mu \in  M(F)$ define  $\phi^o_t(\mu):= \mu\circ T_t$.
 One can check that the family  $\phi^o =(\phi ^o_{t})_{t\geqslant 0}$ is a continuous flow on $M(F)$.
 One can see also  that the transition function of $\phi^o$ is formed from branching kernels on $M(E)$.
 Furthermore, one can compute the weak generator of $\phi^o$ 
 (see assertion (ii) of Example 4.1 from \cite{BaBe16}), it might be considered  a substitute, 
 on the infinite dimensional space $M(E)$, for the gradient type operator 
 $\mathbf{B} \cdot \nabla$ from the above assertion $(i)$.
 \end{exam}
 
 \vspace{1mm}
 
  Let $\tau:F\rightarrow [0,+\infty]$ the first entry time of $\partial D$ by $\phi$, that is,
 $\tau(x)=\inf \{t\geqslant 0:\phi_t(x)\in \partial D\}$. Assume that $\tau$ is bounded, 
 thus there exists $M\in \mathbb{R}$ such that $0\leqslant \tau(x)\leqslant M$ for every $x\in F$. 
 Suppose that  
 \begin{equation}\label{taucont}
 \lim_{D \ni x \rightarrow y }\tau(x)=0 \text{ for every } y\in \partial D. 
\end{equation}
 We consider $\Phi$, the flow $\phi$ stopped  at the boundary of $D$, that is, 
 $\Phi_t=\phi_{t\wedge \tau }$ for $t\geqslant 0$. 
 It follows (cf. \cite{BeBezzCi24} and \cite{BeIoLu})
 that $\Phi$ is  a right continuous flow on $E$. 
 Clearly, its deterministic right Markov process $X$ on $E$ is precisely the
 process $Y$ stopped at the boundary of $D$, $X_t=Y_{t\wedge \tau}$, $t\geqslant 0$.
Let $(T_t)_{t\geqslant 0}$ be the transition function of $X$ (i.e., of $\Phi$),  
$T_tf(x)=f(\Phi_t(x))$ for all $t\geqslant 0$, $x\in E$,  and $f\in \mathcal{B}_+(E)$. 

Denote  by $L$ the weak generator of $\Phi$. 
One can see that $L$ coincides on $D$ with the restriction to $D$ of $\Lambda$, 
the weak generator of $\phi$, in the following sense. 
 If $f\in \mathcal{D}(\Lambda)$
then
$f|_E \in \mathcal{D}(L)$
and $L( f|_E)= \Lambda f$ on $D$ and $Lg=0$ on $\partial D$ for all $g\in \mathcal{D}(L)$.

 For every  Borel measurable function $ f: \partial D \longrightarrow \overline{\mathbb{R}}$ we define  
 $$
 {H}_D f(x) :={ \mathbb{E}}^{x} f (\Phi_{\tau})=f(\phi_{\tau(x)}(x))\mbox{ for all  } 
 x\in {D}.
 $$
 
 Let $(b_{k})_{k\geqslant 1}$ be a sequence of positive numbers such that  $%
 \underset{k\geqslant 1}{\sum }b_{k}\leqslant 1$ and assume that $1<m_{1}:=\underset{k\geqslant 1}{\sum }kb_{k}
 <\infty$. We also fix a function $c:F \rightarrow \mathbb{R}$ such that $c(x)=c_1$ for $x\in D$, where $c_1$ 
 is a constant and $0<c_1 \leqslant \frac{m_{1}}{m_{1}-1}$, and $c(x)=0$ for $x\in F \setminus D$.  

 Let $(T^c_{t})_{t\geqslant 0}$ be the transition
  function of $X$
  killed with the multiplicative functional induced by $c$, i.e.,  for $x\in E$%
 \begin{equation*}
 	T^c_{t}f(x):={\mathbb{E}}^{x}
  \{ e^{-\int_{0}^{t}c(\Phi_s)ds}f(\Phi _{t})\} =e^{-\int_{0}^{t}c(\Phi_s(x))ds}f(\Phi _{t}(x))
 	,\text{ }t\geqslant 0,\text{ 
 	}f\in b\mathcal{B}_+(E). 
 \end{equation*}
 Because $c$ is $0$ on $F \setminus D$,  we have that for $x\in E$, $t\geqslant 0$, and $f\in \mathcal{B}_+(F)$  
 	$$
 	\lim_{t\rightarrow \infty}T^c_{t}(f|_E)(x)=\lim_{t\rightarrow \infty}e^{-\int_{0}^{t}c(\Phi_s(x))ds}f|_E(\Phi _{t}(x))=e^{-c_1\tau(x)}f(\phi _{\tau(x)}(x))=:P^c_{\tau}f(x). 
 	$$
 	
 	 From (\ref{taucont}) and the properties of $\phi$ it follows that for every $f\in b\mathcal{C}_+(\partial D)$ we have 
 	\begin{equation} 
 		\lim_{D\ni x\rightarrow y}P^c_{\tau}f(x)
   =\lim_{D\ni x\rightarrow y}H_D f(x)
   =f(\phi_0(y))=f(y)\text{ for every }y\in\partial D. \label{cont}
 	\end{equation}

  Notice that the validity of 
  (\ref{cont}) means that condition (\ref{taucont}) implies that 
  the stochastic solution to the Dirichlet problem for the weak generator $L$ of $\phi$ 
  is the classical solution, provided that the boundary data is continuous.

    Let $X^{0}=(X_{t}^{0}, \mathbb{P} ^x, x\in E)$ be the trivial Markov process on $E$ for
 which every point is a trap, that is, for each $x\in E$, $\mathbb{P}%
 ^{x}(X_{t}^{0}=x)=1$ for all $t\geqslant 0$, or equivalently, each kernel from
 its transition function $(T_{t}^{0})_{t\geqslant 0}$ is the identity operator, 
 $T_{t}^{0}f=f$ for every $t\geqslant 0$ and $f\in \mathcal{B}_+(E)$. 
 We consider the non-local branching process $\widehat{X^0}$ on $\widehat{E}$ 
 for which the spatial motion is $X^0$ (see \cite{BeVr22}); because actually 
 $\widehat{X^0}$ has no spatial motion, it is called {\it pure branching process}.

    The main assumption is the following "commutation property" of the flow 
    $\Phi$ and the branching mechanism induced by the sequence 
$(B_k)_{k\geqslant 1}$ :
    for all $f\in \mathcal{B}_+(E)$, $f\leqslant 1$,  we have
\begin{equation}
\label{condComp}
    B_k(f \circ \Phi_t)^{(k)}=B_kf^{(k)}\circ \Phi_t,\text{ }k\geqslant 1, \text{ }t\geqslant 0; 
\end{equation}
cf. condition (4.5) from \cite{BeVr22}. 
By assertion $(i)$ of Remark 4.3 from \cite{BeVr22} it follows that
condition (\ref{condComp}) is satisfied by the sequence of kernels $(B_k)_{k\geqslant 1}$ from the above 
Remark \ref{rem2.4} $(v)$.

The following result corresponds to Theorem \ref{thm1.1} in the frame of this subsection. 

\begin{thm} \label{thm3.3}
Let $\varphi \in b\mathcal{B}_+(\partial D)$, $r>0$  be such that 
$\left\vert \left\vert \varphi \right\vert \right\vert _{\infty }<r$, and assume that 
	$$
	\underset{k\geq 1}{\sum }%
	r^{k-1}b_{k} \leqslant 1\text{ and }\underset{k\geq 1}{\sum }kr^{k-1}b_{k}<\infty .
	$$
 We also fix a finite measure $\lambda$ on $D$. 
	 Suppose  that $\tau$ is bounded on $D$, satisfies (\ref{taucont}), and the mapping $%
	[0,\infty )\times E\ni (t,x)\longmapsto \Phi _{t}(x)$ is jointly continuous. 
	Then there
	exist a Borel measurable function 
$g:\partial D\longrightarrow  \overline{\mathbb{R}}_{+}$, 
a $\lambda$-polar $\lambda$-negligible set $M_o\subset D$, and a non-local branching process 
	$\widehat{X}=(\widehat{X}_{t},\widehat{\mathbb{P}}^{\mu}, \mu \in \widehat{E})$ 
	with state space $\widehat{E}$ and spatial motion  $\Phi$,  such that there
	exists 
	\begin{equation}
		\lim_{t\rightarrow \infty }r\widehat{\mathbb{E}}^{\delta _{x}}
  \{ \widehat{\left( \frac{%
				\varphi }{r}\right) }(\widehat{X}_{t})
    \}=\lim_{t\rightarrow \infty }r\widehat{%
			\mathbb{E}^{0}}^{\delta _{x}}
   \{\widehat{\left( \frac{\varphi }{r}\right) }%
		(\Phi _{t}(\widehat{X^{0}_{t}}))\} =:u(x)\text{ for all }x\in E.  
		\label{solutieFlow}
	\end{equation}%
	The function $u$ is a generalized solution to the nonlinear Dirichlet problem (\ref{problem2}) 
	with boundary data $\varphi$ and $L$, the weak generator of $\Phi$, that is%
	\begin{equation} \label{ProbFlow}
		\left\{ 
		\begin{array}{l}
			(L-c)u+c\sum_{k\geqslant 1}b_{k}B_{k}u^{(k)}=0\text{ in }D \\[3mm]
				u\overset{k}{\longrightarrow } \varphi \text{ on }D \setminus M_o,%
		\end{array}%
		\right.		
	\end{equation}%
	where $k=H_Dg$ is $\lambda$-integrable 
	and the exceptional $\lambda$-polar  $\lambda$-negligible  set is 
 $M_o=[k= +\infty]$. 
  
\end{thm}

\begin{proof}
We follow the steps of the proof of Theorem \ref{thm1.1}.
	We treat first the case  $r=1$. 
Let $( H_t f)_{t\geqslant 0}$
be the unique solution of the equation (\ref{Ht}) and
$(\widehat{\mathbf{H}}_t)_{t\geqslant 0}$ be the induced branching semigroup of kernels on $\widehat{E}$ such that 
 $\widehat{\mathbf{H}}_t \widehat{f}=\widehat{H_t f}$ for all $f \in b\mathcal{B}_+(E)$, $ f \leqslant 1$. 
 By (C) there exists a branching Markov process $\widehat{X}$ with state space $\widehat{E}$ 
 having the transition function $(\widehat{\mathbf{H}}_t)_{t\geqslant 0}$.
 
 Employing Proposition 2 from \cite{Naga76} we show that for every $x\in E$ the limit 
	$u(x):=$ $\lim_{t\rightarrow \infty}H_{t}\varphi(x) $ exists, 
	using also some arguments from the proof of Theorem 2.1 in \cite{LuSt17}. 
 In order to do so, we verify  condition (6) from \cite{Naga76}, namely that for some $\varepsilon >0$ we have 
	$\mathbb{P}^x (\tau < \zeta)\geqslant \varepsilon \text{ for all }x\in E$, 
	where $\zeta$ is the killing time of the process $\Phi_t$ with the multiplicative functional induced by $c$. 
 Notice that the argument from  \cite{Naga76} 
 to obtain the above inequality is inadequate here because  the transition function of a flow is not strong Feller.
 However, we can argue as follows: we have that 
 ${\mathbb{P}}^x (\tau < \zeta) ={\mathbb{E}}^x\{e^{-\int_{0}^{\tau} c(\Phi_s)ds}\} = {\mathbb{E}}^x\{ e^{-\tau c_1}\} =
	e^{-\tau(x) c_1}\geqslant e^{-M c_1}=:\varepsilon >0 $, where
 for the proof of the first equality we send to  \cite{MeTh}, page 143, and (62.22) from \cite{Sharpe}.

  As in the proof of Theorem \ref{thm1.1} we have that $u$ belongs to the domain of $L$ 
  and  it satisfies the equation $(L-c)u+c\sum_{k\geqslant 1}b_{k}B_{k}u^{(k)}=0$  in $D$, and
  $%
	 u(x)=P^c_{\tau}\varphi+w,$ where 
	 $w:=\lim_{t\rightarrow \infty }$
  $\int_{0}^{t}T^c_s(c\sum_{k\geqslant 1}b_{k}B_{k}(
	 H_{t-s}\varphi) ^{(k)}) ds.$ 
  It follows that 
$w(x)\leqslant \int_{0}^{\infty }T^c_{s}c(x)ds\leqslant \mathbb{E}^{x }\int_{0}^{\tau }c(\phi
_{s})ds=\int_{0}^{\tau (x)}c_{1}ds=c_{1}\tau (x)$. By (\ref{taucont}) we have
  $\lim_{D\ni x\rightarrow y}w(x)=0$
	 for all $y\in \partial D$.

	  It follows from (\ref{cont}) and Lemma \ref{Lema} that there exists $g \in \mathcal{B}_+(\partial D)$ such that  $P_{\tau }^{c}\varphi\overset{k_o}{\longrightarrow }\varphi$ on $[k_o<\infty]$ and $k_o:= P_{\tau}^cg$.
   Therefore $u\overset{k}{\longrightarrow }\varphi$ on $[k<\infty]$, where $k=H_D g= e^{c_1\tau} P^ c_\tau g\geqslant k_o$.
   According with Theorem 4.2 from \cite{BeVr22}, 
   we have that the branching process $\widehat{X}$ has the representation 
   $\widehat{X}_t =\Phi_t(\widehat{X^0_t})$, $t\geqslant 0$, where  the equality is in the distribution sense.

   The case  $r>1$ is treated similarly to the second part of the proof of Theorem \ref{thm1.1}, 
   the solution $u$ is given by (\ref{solutieFlow}).  
\end{proof}

\noindent
{\bf Examples of continuous flows leaving a bounded open set and satisfying (\ref{taucont}) }  

(1) Let $F=[0,\infty)$, $D=[0,1)$, and $\psi=(\psi_t)_{t\geqslant 0}$ be the uniform motion to the right, 
$\psi_t(y)=y+t$, $y \in F$, $t\geqslant 0$. If $\tau$ is the entry time of $\partial D=\{1\}$, 
then clearly $\tau(x)=1-x$ for $x \in F$, $\tau\leqslant diam(D)=1$, and $\tau$ satisfies (\ref{taucont}).

(2)
Let $\phi^o=(\phi^o_t)_{t\geqslant 0}$ be a continuous flow on $F^{d-1}$, $d\geqslant 2$, 
and consider the flow $\phi=(\phi_t)_{t\geqslant 0}$ on $F^{d}$ given by the Cartesian product of $\phi^o$ and $\psi$, 
$$
\phi_t (x,y)=(\phi^o_t(x),\psi_t(y)), (x,y)\in F^d=F^{d-1}\times F, t\geqslant 0.
$$ 
Then for every bounded open set $D\in F^d$ the entry time $\tau$ of $\partial D$ is bounded. 
Indeed, let $m:=diam(D_1)$, where $D_1$ is the projection of $D$ on $F$. 
Let $\tau^1:F\longrightarrow \mathbb{R}_+\cup \{+\infty\}$ be the first entry time of $\partial D_1$ by $\psi$. 
Then $\tau^1(y)\leqslant m$ for every $y\in D_1$ and clearly $\tau((x,y))\leqslant \tau^1(y)\leqslant m$ for every $(x,y)\in D$.

(3) 
Assume that $d=2$, 
$D:= [0,1)^2$, and that $\phi^o=\psi$.
It follows from (2) that
$\tau^2$, the entry time of the boundary of $D$ by $\phi$, is bounded and in addition
$\tau^2((x,y)) \leqslant \inf(\tau(x), \tau(y))$ for any $(x,y)\in D$.
By (1) we conclude that 
$\tau^2$ also satisfies (\ref{taucont}).

\vspace{2mm}

\noindent{\bf Acknowledgements.} 
This work was supported by a grant of the Ministry of Research, Innovation and
Digitization, CNCS - UEFISCDI, project number PN-III-P4-PCE-2021-0921, within PNCDI III.

\section{Appendix}
\noindent
{\bf Fine topology.} Let $F$ be a Lusin topological space and $\mathcal{B}(F)$ its Borel $\sigma$-algebra. Let $\mathcal{U}=(U_{\alpha})_{\alpha >0}$ be the sub-Markovian resolvent of kernels on $(F, \mathcal{B}(F))$ 
of the right Markov process 
$Y=(Y_t, \Omega,  \mathcal{F}, \mathcal{F}_t, X_t, \theta_t, \mathbb{P}^x, \xi)$ 
with transition function $(T_t)_{t\geq 0}$, that is, $U_{\alpha}f:= \int^{\infty}_0 e^{-\alpha t}T_tf dt$, $f\in \mathcal{B}_+(F)$. 

A function $u\in \mathcal{B}_+(F)$ is called 
{\it $\mathcal{U}$-excessive} provided that $\alpha U_{\alpha}\leq u$ for all $\alpha >0$,
and $\lim_{\alpha \rightarrow \infty} \alpha U_{\alpha}u(x)=u(x)$, $x\in F$. 
If $\beta >0$, we denote by $\mathcal{U}_{\beta}$ 
the sub-Markovian resolvent of kernels defined as $\mathcal{U}_{\beta}=(U_{\beta + \alpha })_{\alpha>0}$. 
The \textit{fine topology} is the coarsest topology on $F$ making continuous all 
$\mathcal{U}_{\beta}$-excessive functions for some (and equivalently for all) $\beta >0$. 
Recall that a function $f\in \mathcal{B}_+(F)$  is finely continuous if and only if 
$t\longmapsto f(Y_t)$ is a.s. right continuous on $[0, \xi)$.
Consequently, the given Lusin topology on $F$ is smaller than the fine topology.

\vspace{1mm}

\noindent
{\bf Weak generator.} 
Following \cite{Fi89}, page 354,
and \cite{BeBezzCi24}, 
define 
$$
\mathcal{B}_o:=\{ f \in b\mathcal{B}(E): f \textit{ is finely continuous  on } E \textit{ with respect to } Y \}
$$
and
$$
\begin{aligned}
\mathcal{D}(L):=\{ f \in \mathcal{B}_o : \left ( \frac{\mathbb{E}^x f(Y_t)-f(x)}{t} \right )_{t,x} \textit{ is bounded for } x \in E \textit{ and } t \textit{ in a neighbourhood of 0} \\ \textit{ and there exists } 
\lim_{t \rightarrow 0} \frac{\mathbb{E}^xf(Y_t)-f(x)}{t} \textit{ pointwise (in $x$) and the above limit is a function from } \mathcal{B}_o  \}.
\end{aligned}
$$
For $f \in \mathcal{D}(L)$ and $ x\in E$ let
$$
Lf(x):=\lim_{t \rightarrow 0} \frac{\mathbb{E}^x f(Y_t)-f(x)}{t}.
$$

We have that $L:\mathcal{L}(D) \longrightarrow b\mathcal{B}(E)$ is a linear operator, 
$\mathcal{D}(L)=U_\alpha(\mathcal{B}_o)$, $\alpha>0$, and if $f=U_\alpha g$, $g \in \mathcal{B}_o$, then $(\alpha -L)f=g$. 
The pair $(L, \mathcal{D}(L))$ is called the \textit{weak generator} of the process $X$.

\bigskip

\noindent
{\bf Non-local branching processes.} 
Let $p_{1}$ and $p_{2}$ two finite measures on $\widehat{E}.$ Recall that
their convolution $p_{1}\ast p_{2}$ is the finite measure on $\widehat{E}$
defined for every $F\in b\mathcal{B}_+(\widehat{E})$ by 
\begin{equation*}
	\int_{\widehat{E}}p_{1}\ast p_{2}\left( d\nu \right) F\left( \nu \right)
	:=\int_{\widehat{E}}p_{1}\left( d\nu_{1}\right) \int_{\widehat{E}}p_{2}\left(
	d\nu _{2}\right) F\left( \nu _{1}+\nu _{2}\right) .
\end{equation*}

Recall that (see e.g. \cite{Si68}) 
a kernel $N$ on $(\widehat{E},\mathcal{B}(\widehat{%
	E}))$ which is sub-Markovian (i.e. $N1\leqslant 1)$ is called {\it branching kernel} 
provided that for all $\mu ,\nu \in \widehat{E}$ we have $N_{\mu +\nu
}=N_{\mu }\ast N_{u},$ where $N_{\mu }$ denotes the measure on $\widehat{E}$
such that $\int_{\widehat{E}} gdN_{\mu }=Ng\left( \mu \right) $ for all $g\in b\mathcal{B}_+%
(\widehat{E}).$

A Markov process with state space $\widehat{E}$ is a branching process if and only if its transition function
is formed from branching kernels (see e.g. \cite{Li22} and \cite{BeLu16}; see also 
\cite{BeLuVr20} and \cite{BeDeLu15}).

\vspace{1mm}

\begin{proof}[\bf{Proof of the equality $(\ref{eq2.1})$}] 
Let $\omega \in \Omega$ be  such that $\tau(\omega)<\infty$.
If $t>\tau(\omega)$
then 
$\int_0^t c(X_s(\omega)) ds= \int_0^{\tau(\omega)}c(X_s(\omega)) ds+
\int_{\tau(\omega)}^t c(X_s(\omega))ds=
\int_0^{\tau(\omega)}c(Y_s(\omega)) ds$,
where the last equality holds because 
$X_s= Y_\tau\in \partial D$ on $[s\geqslant \tau]$ and $c$ vanishes on $\partial D$.
Since we assumed that 
for every $x\in D$ we have 
$\mathbb{P}^x$-a.s. that 
$\tau<\infty$, 
it follows that
$\mathbb{P}^x$-a.s. there exists
$\lim_{t\to \infty} 
e^{-\int_0^t c(X_s) ds} f(X_t)=
e^{-\int_0^\tau c(Y_s) ds} f(Y_\tau)$.
By dominate convergence we conclude now that $(\ref{eq2.1})$ holds.
\end{proof}

\vspace{1mm}

\begin{proof}[\bf{Proof of Lemma \ref{lemma2.1}}] 
Let $f\in b\cb_+(F)$ and $k_t := S_t^c f$, $t\geqslant 0$.
By Proposition 3.3 from \cite{BeJEMS} it follows that $(k_t)_{t\geqslant 0}$ is the unique solution of the integral equation
$$
k_t= S_tf -
\int_0^t S_s(c k_{t-s})  ds, \, t\geqslant 0.
$$
By hypothesis the functions
$S_tf$ and $S_s(c k_{t-s})$, $0<s< t$,
are continuous on $F$ and the dominate convergence implies  that the function $\int_0^t S_s(c k_{t-s})  ds$ is also continuous. 
Hence $k_t$ is a continuous function for every $t>0$, completing the proof of assertion $(i)$.
If in addition  
$S_tf$ vanishes at infinity on  the  locally compact space  $F$,   
then clearly the continuous function $S^c_tf$ 
also vanishes at infinity because $S^c_tf\leqslant S_tf$.
\end{proof}

\vspace{1mm}

\begin{proof}[\bf{Proof of Proposition \ref{prop2.2}}] 
We consider the process obtained from $Y$ by killing with the multiplicative functional induced by $c$. 
Its transition function is $(S_t^c)_{t\geqslant 0}$ and by Lemma \ref{lemma2.1} it is strong Feller and 
acts on continuous functions vanishing at infinity. 
On the other hand by Theorem 6.12 from \cite{BlGe68}
(Hunt's fundamental theorem) it follows that
$P_\tau^c$ coincides on $D$ with the hitting kernel of $\partial D$ with respect to the killed process.
Since we assumed that the potential kernel $V$ is proper, 
we can apply now Theorem 13.1 from \cite{Dynkin} 
with respect to the above mentioned killed process and we deduce that $(\ref{continuous})$ holds.
\end{proof}

\vspace{1mm}

\begin{proof}[\bf{Sketch of the proof of Lemma \ref{Lema}}] 
Recall that $P_\tau^c$ coincides on $D$ with the hitting kernel of $\partial D$ with respect to the killed process 
(see e.g. the proof of Proposition \ref{prop2.2}). 
Let %
\begin{equation*}
	\begin{array}{cc}
		\mathcal{M} & :=\mathcal{\{}f\in L_{+}^{1}( \partial D,\sigma ) :%
		\text{ there exists }g\in \mathcal{B}_+( \partial D) \text{ such
			that }P_{\tau }^{c}f\overset{k}{\rightarrow }f\text{ on } \\ 
		& \text{ }\left[ k\,<\infty \right] \text{ where }k:=P_{\tau}^cg\in
		L_{+}^{1}\left( D,\lambda \right) \}.%
	\end{array}%
\end{equation*}%
If $f\in \mathcal{C}_{+}( \partial D)$, taking $k=0$, 
we have from (\ref{continuous}) that $\mathcal{C}_{+}\left( \partial D\right)
\subset \mathcal{M}.$ It is sufficient to prove that 
\begin{equation}
	\text{if }\left( f_{n}\right) _{n}\subset \mathcal{M}\text{, }f_{n}\nearrow
	f\in L_{+}^{1}\left( \partial D,\sigma \right) \text{ then }f\in \mathcal{M}%
	\text{.}  \label{bogdanovici}
\end{equation}%
Indeed, we can argue as in the proof of Theorem 5.3 from \cite{BeCoRo11}.
Notice that the control function $k$ is hyperharmonic on the
domain $D$ and thus the set $M:=\left[ k=\infty \right] $ is $\lambda$-polar.

To prove (\ref{bogdanovici}) one can use the technique 
in the proof of Theorem 4.8 from \cite{Be11} (see Theorem 5.3 from \cite{BeCoRo11}).	
\end{proof}

\end{document}